\documentclass[12pt,notitlepage]{amsart}

\usepackage{graphicx}
\usepackage{amsmath, amsfonts, amssymb, amsthm}
\usepackage[myheadings]{fullpage}
\usepackage[utf8]{inputenc}
\usepackage[margin=1in]{geometry}
\usepackage{enumitem}
\usepackage{xcolor}
\usepackage{caption}
\usepackage{tabularx}
\usepackage{fancyhdr}
\usepackage{comment}
\newtheorem*{corollary*}{Corollary}
\newtheorem*{remark*}{Remark}
\usepackage{hyperref}
\usepackage{tikz}
\usetikzlibrary{angles,quotes}
\usepackage{mdframed}
\usepackage{lipsum}
\usepackage{chngcntr}
\usepackage{imakeidx}
\usepackage{tkz-euclide}
\usepackage{blindtext}
\usepackage[inline]{asymptote}
\usepackage{mathdots}
\usepackage{tikz-cd,mathtools}
\usepackage{tikz-3dplot}
\usepackage{ mathrsfs }
\usepackage{float}
\restylefloat{table}

\usepackage{pgfplots, pgfplotstable}
\usepackage{multicol}
\usepackage{endnotes}
\usepackage{idxlayout}
\usepackage{xcolor, forest}
\usepackage{tocvsec2}
\usepackage[toc]{glossaries}
\usepackage{venndiagram}
\usepgfplotslibrary{statistics}
\usepackage[colorinlistoftodos]{todonotes} 
\usepackage[normalem]{ulem}
\usetikzlibrary{shapes.geometric}
\usepackage{mathtools}

\numberwithin{equation}{section}

\newtheoremstyle{thmstyleone}%
  {3pt}
  {3pt}
  {\itshape}
  {}
  {\bfseries}
  {.}
  { }
  {}

\theoremstyle{thmstyleone}%
\newtheorem{theorem}{Theorem}[section]
\newtheorem{lemma}[theorem]{Lemma}
\newtheorem{proposition}[theorem]{Proposition}%
\newtheorem{corollary}[theorem]{Corollary}%

\theoremstyle{thmstyletwo}%
\newtheorem{remark}[theorem]{Remark}%
\newtheorem{example}[theorem]{Example}%

\theoremstyle{thmstylethree}%
\newtheorem{definition}[theorem]{Definition}%

\numberwithin{theorem}{section}
\numberwithin{equation}{section}

\begin{document}

\title{On Weighted and Bounded Multidimensional Catalan Numbers}

\author{Ryota Inagaki}
\address{Department of Mathematics, Massachusetts Institute of Technology}
\email{inaga270@mit.edu}

\author{Dimana Pramatarova}
\address{"Akademik Kiril Popov" High school of Mathematics}
\email{dimanapramatarova@gmail.com}

\date{\today}
\begin{abstract}We define a weighted analog for the multidimensional Catalan numbers, obtain matrix-based recurrences for some of them, and give conditions under which they are periodic. Building on this framework, we introduce two new sequences of triangular arrays: the first one enumerates the $k$-dimensional Balanced ballot paths of exact height $s$; the second one is a new multidimensional generalization of the Narayana numbers, which count the number of Balanced ballot paths with exactly $p$ peaks.\end{abstract} 

\maketitle

\section{Introduction}
\label{section:intro}
The sequence of the Catalan numbers $\displaystyle C_n = \frac{1}{n+1} \binom{2n}{n}$ is one of the most studied ones in the field of enumerative combinatorics, which is the branch of mathematics dedicated to counting discrete structures by deriving exact formulas, generating functions, or recursive relations. The Catalan numbers (sequence $A000108$ in the OEIS \cite{oeis}) enumerate various objects such as the triangulations of a convex polygon with $n+2$ sides, rooted binary trees with $2n$ nodes, along with hundreds of others \cite{StanleyCatalan}. More notably, they count the number of Dyck paths of length $2n$, which are sequences of points in $\mathbb{Z}^2$ starting at $(0,0)$ and ending at $(2n,0)$, composed of $n$ up-steps of $(1,1)$ and $n$ down-steps of $(1, -1)$, and the paths do not go below the $x$ axis (see Figure \ref{fig:dyckpath} for an example).
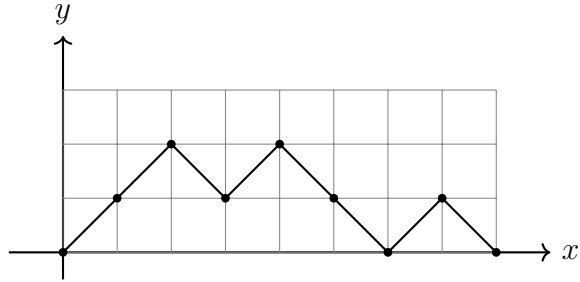
\begin{figure}[h]
\centering
\begin{tikzpicture}[scale=0.72]
  \draw[->, thick] (-1, 0) -- (9, 0) node[right] {$x$};
  \draw[->, thick] (0, -0.5) -- (0, 4) node[above] {$y$};

  \draw[step=1,gray,very thin] (0,0) grid (8,3);

  \draw[thick,black] 
    (0,0) -- 
    (1,1) -- 
    (2,2) -- 
    (3,1) -- 
    (4,2) -- 
    (5,1) -- 
    (6,0) --
    (7,1) --
    (8,0);
  \foreach \x/\y in {0/0,1/1,2/2,3/1,4/2,5/1,6/0,7/1,8/0}
    \filldraw[black] (\x,\y) circle (2pt);
\end{tikzpicture}
\vspace{-0.5 em}
\caption{A Dyck path of $8$ steps}
    \label{fig:dyckpath}
\end{figure}

Now, we introduce weighted Catalan numbers, which first appeared in works such as those of Goulden and Jackson \cite{Goulden}. For fixed sequence of integers $\vec{b} = (b_0,b_1,b_2,\ldots) \in \mathbb{Z}^{\mathbb{N}}$, which we call \emph{weight vector}, and a Dyck path $P$ of length $2n$, the \emph{weight} $wt_{\vec{b}}(P)$ of the Dyck path $P$ is the product $b_{h_1} b_{h_2}\cdots b_{h_n}$, where $h_i$ is the height of the starting point of the $i$-th up-step of $P$. The corresponding \emph{$n$-th weighted Catalan number for $\vec{b}$} is defined as $\displaystyle C_n^{\vec{b}} = \sum_P wt_{\vec{b}}(P),$  where the sum is over all Dyck paths of length $2n$. Examples of weighted Dyck paths are displayed in Figures~\ref{fig:weighted-dyck} and ~\ref{fig:2ddycklength3}.
\vspace{-0.5em}

\begin{figure}[H]
\centering
\begin{tikzpicture}[scale=0.72]

  \draw[step=1cm, gray!30, very thin] (0,0) grid (8,3);

  \draw[->] (-1,0) -- (8.5,0) node[right] {$x$};
  \draw[->] (0,-0.5) -- (0,3.5) node[above] {$y$};

  \coordinate (P0) at (0,0);
  \coordinate (P1) at (1,1);
  \coordinate (P2) at (2,2);
  \coordinate (P3) at (3,1);
  \coordinate (P4) at (4,2);
  \coordinate (P5) at (5,1);
  \coordinate (P6) at (6,0);
  \coordinate (P7) at (7,1);
  \coordinate (P8) at (8,0);

  \draw[black, thick] 
    (P0) -- (P1) -- (P2) -- (P3) -- (P4) -- (P5) -- (P6) -- (P7) -- (P8);

  \foreach \pt in {P0,P1,P2,P3,P4,P5,P6,P7,P8}
    \fill[black] (\pt) circle (2pt);

  \node[black] at (0.3,0.7) {$b_0$}; 
  \node[black] at (1.3,1.7) {$b_1$}; 
  \node[black] at (3.3,1.7) {$b_1$}; 
  \node[black] at (6.3,0.7) {$b_0$}; 

\end{tikzpicture}
\caption{A weighted Dyck path with $wt_{\vec{b}}(P) = b_0^2b_1^2$}
\label{fig:weighted-dyck}
\end{figure}
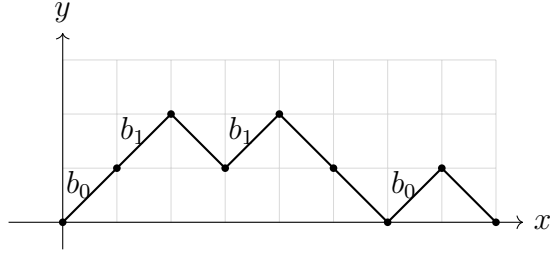

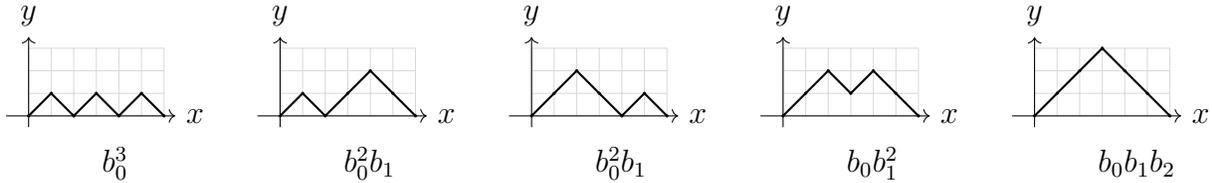
\begin{figure}[H]
    \centering
\begin{minipage}{0.19\linewidth}
\centering
\begin{tikzpicture}[scale=0.3]
  \draw[step=1cm, gray!30, very thin] (0,0) grid (6,3);
  \draw[->] (-1,0) -- (6.5,0) node[right] {$x$};
  \draw[->] (0,-0.5) -- (0,3.5) node[above] {$y$};
  \coordinate (P0) at (0,0);
  \coordinate (P1) at (1,1);
  \coordinate (P2) at (2,0);
  \coordinate (P3) at (3,1);
  \coordinate (P4) at (4,0);
  \coordinate (P5) at (5,1);
  \coordinate (P6) at (6,0);
  \draw[black, thick] (P0)--(P1)--(P2)--(P3)--(P4)--(P5)--(P6);
  \foreach \pt in {P0,P1,P2,P3,P4,P5,P6}
    \fill[black] (\pt) circle (2pt);
\end{tikzpicture}
\vspace{-0.6 em}
\captionof*{figure}{$b_0^3$}
\end{minipage}
\hfill
\begin{minipage}{0.19\linewidth}
\centering
\begin{tikzpicture}[scale=0.3]
  \draw[step=1cm, gray!30, very thin] (0,0) grid (6,3);
  \draw[->] (-1,0) -- (6.5,0) node[right] {$x$};
  \draw[->] (0,-0.5) -- (0,3.5) node[above] {$y$};
  \coordinate (P0) at (0,0);
  \coordinate (P1) at (1,1);
  \coordinate (P2) at (2,0);
  \coordinate (P3) at (3,1);
  \coordinate (P4) at (4,2);
  \coordinate (P5) at (5,1);
  \coordinate (P6) at (6,0);
  \draw[black, thick] (P0)--(P1)--(P2)--(P3)--(P4)--(P5)--(P6);
  \foreach \pt in {P0,P1,P2,P3,P4,P5,P6}
    \fill[black] (\pt) circle (2pt);
\end{tikzpicture}
\vspace{-0.6 em}
\captionof*{figure}{$b_0^2b_1$}
\end{minipage}
\hfill
\begin{minipage}{0.19\linewidth}
\centering
\begin{tikzpicture}[scale=0.3]
  \draw[step=1cm, gray!30, very thin] (0,0) grid (6,3);
  \draw[->] (-1,0) -- (6.5,0) node[right] {$x$};
  \draw[->] (0,-0.5) -- (0,3.5) node[above] {$y$};
  \coordinate (P0) at (0,0);
  \coordinate (P1) at (1,1);
  \coordinate (P2) at (2,2);
  \coordinate (P3) at (3,1);
  \coordinate (P4) at (4,0);
  \coordinate (P5) at (5,1);
  \coordinate (P6) at (6,0);
  \draw[black, thick] (P0)--(P1)--(P2)--(P3)--(P4)--(P5)--(P6);
  \foreach \pt in {P0,P1,P2,P3,P4,P5,P6}
    \fill[black] (\pt) circle (2pt);
\end{tikzpicture}
\vspace{-0.6 em}
\captionof*{figure}{$b_0^2b_1$}
\end{minipage}
\hfill
\begin{minipage}{0.19\linewidth}
\centering
\begin{tikzpicture}[scale=0.3]
  \draw[step=1cm, gray!30, very thin] (0,0) grid (6,3);
  \draw[->] (-1,0) -- (6.5,0) node[right] {$x$};
  \draw[->] (0,-0.5) -- (0,3.5) node[above] {$y$};

  \coordinate (P0) at (0,0);
  \coordinate (P1) at (1,1);
  \coordinate (P2) at (2,2);
  \coordinate (P3) at (3,1);
  \coordinate (P4) at (4,2);
  \coordinate (P5) at (5,1);
  \coordinate (P6) at (6,0);

  \draw[black, thick] 
    (P0) -- (P1) -- (P2) -- (P3) -- (P4) -- (P5) -- (P6);

  \foreach \pt in {P0,P1,P2,P3,P4,P5,P6}
    \fill[black] (\pt) circle (2pt);
\end{tikzpicture}
\vspace{-0.6 em}
\captionof*{figure}{$b_0b_1^2$}
\end{minipage}
\hfill
\begin{minipage}{0.19\linewidth}
\centering
\begin{tikzpicture}[scale=0.3]
  \draw[step=1cm, gray!30, very thin] (0,0) grid (6,3);
  \draw[->] (-1,0) -- (6.5,0) node[right] {$x$};
  \draw[->] (0,-0.5) -- (0,3.5) node[above] {$y$};
  \coordinate (P0) at (0,0);
  \coordinate (P1) at (1,1);
  \coordinate (P2) at (2,2);
  \coordinate (P3) at (3,3);
  \coordinate (P4) at (4,2);
  \coordinate (P5) at (5,1);
  \coordinate (P6) at (6,0);
  \draw[black, thick] (P0)--(P1)--(P2)--(P3)--(P4)--(P5)--(P6);
  \foreach \pt in {P0,P1,P2,P3,P4,P5,P6}
    \fill[black] (\pt) circle (2pt);
\end{tikzpicture}
\vspace{-0.6 em}
\captionof*{figure}{$b_0b_1b_2$}
\end{minipage}
   \caption{For weight vector $\vec{b}= (b_0,b_1,b_2)$, all $5$ weighted Dyck paths of $6$ steps with corresponding weights for weight vector $\vec{b} = (b_0,b_1, b_2, \dots)$. The third weighted Catalan number for this weight vector is $C^{\vec{b}}_3 = b_0^3 + 2b_0^2b_1 + b_0b_1^2+ b_0b_1b_2$.}
    \label{fig:2ddycklength3}
    \end{figure}


 For particular weight vectors, the weighted Catalan numbers have many combinatorial interpretations. For example, when the weight vector is $\vec{b} = (1, q, q^2, \dots)$, the corresponding weighted Catalan number $C_n^{\vec{b}}$ is the $q$-Catalan numbers \cite{qcatalan}, which encode the distribution of areas under Dyck paths. Postnikov \cite{postnikov2000counting} proved that when the weight vector is set to be $\vec{b} = (1^2, 3^2, 5^2, \dots)$, the weighted Catalan number $C^{\vec{b}}_n$ counts combinatorial types of Morse links of order $n$. Postnikov conjectured that $C^{\vec{b}}_n$ has a period of $2\cdot3^{r-3}$ modulo $3^r$, meaning that $2\cdot 3^{r-3}$ is the smallest positive integer such that $C_{n+2\cdot 3^{r-3}}^{\vec{b}}- C_{n}^{\vec{b}}$ is a multiple of $3^r$ for large $n$. This was later proven by Gao and Gu \cite{gaogu} in $2021$. 

Arithmetic properties of the weighted Catalan numbers have also been extensively studied. In $2006$, Postnikov and Sagan \cite{postnikov2007power} derived a condition under which the $2$-adic valuation of the weighted Catalan numbers is equal to that of the corresponding unweighted ones. Later in the year, Konvalinka \cite{konvalinka} proved an analogous result for the $q$-Catalan numbers. In $2010$, An \cite{ansthesis} proved a conjecture by Konvalinka and studied other divisibility properties using matrices. Later, in $2012$, Shader \cite{shader2013weighted} considered the periodicity modulo $p^r$ for prime $p$ of certain weighted Catalan numbers. In 2021, Gao and Gu proved a condition for the periodicity of the weighted Catalan numbers modulo an integer \cite[Theorem~4.2]{gaogu}.


We build on the above results by extending the definition of weighted two-dimensional Catalan numbers to \textit{weighted multidimensional} ones by considering height functions of similar behavior. The paper is organized as follows. In Section~\ref{section:preliminaries}, we define Balanced ballot paths, the generalization of Dyck paths to higher dimensions, and use them to present our generalization of weighted Dyck paths to higher dimensions. We discuss prior and basic results on the 2-dimensional Catalan numbers and prove Gao and Gu’s Theorem \cite{gaogu} using a matrix-based approach; this inspires our techniques in deriving periodicity properties in the $k$-dimensionalCatalann numbers. In Section~\ref{sec:periodicity}, we discuss properties guaranteeing periodicities of weighted $k$-dimensional Catalan numbers modulo $m$, where $m$ is an integer. In Section~\ref{sec:multidimensional}, we use recurrence sequences of integers to find closed-forms for some cases of the multidimensional bounded and weighted Catalan numbers. Using the bounded multidimensional Catalan numbers, in Section~\ref{sec:heighttriangles} we construct new integer sequences, the multidimensional triangles of Balanced ballot paths of height exactly $,s$ and establish properties. In Section~\ref{sec:Narayana}, we use our definition of height to define peaks and also consider the number of peaks in ballot paths to construct analogs of the Narayana numbers. Code used to calculate the 3 a4-dimensionalnal Balanced-Ballot-Path-Height triangles and the 3 a4-dimensionalnal Narayana triangles can be found on the GitHub repository \href{https://github.com/Ryota-IMath/Inagaki_Pramatarova_multidim_height_Catalan}{https://github.com/Ryota-IMath/Inagaki\textunderscore Pramatarova\textunderscore multidim \textunderscore height\textunderscore Catalan}.

\section{Definitions and Notations}
\label{section:preliminaries}

\subsection{Problem Setup}
 
We begin by discussing variants of the Dyck path and their extensions to higher dimensions. Consider the east-north version \cite{StanleyCatalan}, which is also known as the $2$-dimensional ballot path with $n$ east steps and $n$ north steps. By scaling, rotating, and flipping the path, one sees that the definitions of Dyck paths and ballot paths ending at $(n,n)$ are equivalent. 

\begin{definition}
    A \textbf{(2-dimensional) Balanced ballot path} of $2n$ steps is a sequence of points in 
    $\mathbb{Z}^2$, starting at $(0,0)$ and ending at $(n,n)$, formed from $n$ east‐steps $(1,0)$ and $n$ north‐steps $(0,1)$, so that the path never goes above the diagonal $y = x$.
\end{definition}

We now extend this to higher dimensions. To the best of our knowledge, the generalization of the multidimensional weighted Catalan numbers has not been previously defined in the literature. A \textit{point} in the $k$-dimensional lattice $\mathbb{Z}^k$ is a $k$-tuple $(x_1, x_2, \ldots, x_k)$, and steps are taken in the positive coordinate directions, typically along the standard basis vectors $\vec{e_i} = (0,0,\ldots, 0, 1, 0, \ldots, 0)$, in which the $i$-th coordinate is $1$ and the others are $0$. 

\begin{definition}
    The $k-$\textbf{dimensional Balanced ballot path} of $kn$ steps, denoted as $P_{k, n} = v_1, v_2, v_3, \dots, v_{kn}$, is a sequence of $kn$ steps in $\mathbb{Z}^k$ starting at $(0,0,\ldots,0)$ and ending at $(n,n,\ldots,n)$ satisfying the following conditions: 
    \begin{itemize}
        \item Each step $v_i - v_{i-1}$ is in the set of standard unit vectors $\{\vec{e}_1, \vec{e}_2, \dots, \vec{e}_k\}$.
        \item Each point $x = (x_1,\ldots,x_k)$ in the path satisfies $x_1\geq x_2 \geq \ldots \geq x_k$.
    \end{itemize}
    
    We call an \textbf{up-step} any step in the direction of $\vec{e_1} = (1,0,\ldots,0)$. 
\end{definition}

Note that ballot paths do not require an equal number of steps in each direction; however, we consider the balanced case, which is essentially the same as multidimensional Dyck paths, but we use this term to distinguish them from the definition of Dyck paths presented in the introduction. Using these, we can define the $k$-dimensional Catalan number:
\begin{definition}[($A060854$ in OEIS \cite{oeis})]
    For $n$ and $k$, the $n$th $k$-dimensional Catalan number is the number of $k$-dimensional Balanced ballot paths of $kn$ steps. The $n$-th $k$-dimensional Catalan number equals $$C_{k,n} = \frac{0!1!\cdots(n-1)! \cdot (kn)!}{k!(k+1)!\cdots(k+n-1)!}.$$
\end{definition}
\begin{remark}
    One can observe from the above formula that $C_{k, n}= C_{n, k}$.
\end{remark}

\begin{remark}
    The $n$th $k$-dimensional Catalan number is the number of Standard Young Tableaux $k\times n$, as derived by the hook length formula \cite{MR4621625}.
\end{remark}

We now extend the notion of bounded and weighted Catalan numbers to $k$-dimensional Catalan numbers. In order to define them, we define a height function as follows.

\begin{definition}
    \label{defn:height-manhattan}
    For point $x \in \mathbb{Z}^k$, we define \textbf{height} of $x$  as $$h_k(x) := x_1 - x_2 + x_1 -x_3 +\ldots +x_1 - x_k = (k-1)x_1 - \sum_{i=2}^k x_i.$$

    Given a $k-$dimensional Balanced ballot path $P$, define the \textbf{height of the path $P$} to be $\max \{h_k(x): x \in P \}$.
\end{definition}

This is a natural extension of the $2$-dimensional Catalan numbers, as the height is the difference between the number of up-steps and down-steps, i.e., $x_1 - x_2$. The height function is the Manhattan distance \cite{manhattan_distance} from point $(x_1,x_2,\ldots,x_k)$ to $(x_1,x_1,\ldots, x_1)$.

\begin{example} An example of a $3$-dimensional Balanced ballot path is shown in Figure \ref{fig:3ddyck}, where the black arrows correspond to $\vec{e_1}$, the gray arrows to $\vec{e_2}$ and the light gray ones to $\vec{e_3}$, with the values of the height at the points where it increases - at each $\vec{e_1}$ step. 
    \begin{figure}[H]
        \centering
        \includegraphics[width=0.6\linewidth]{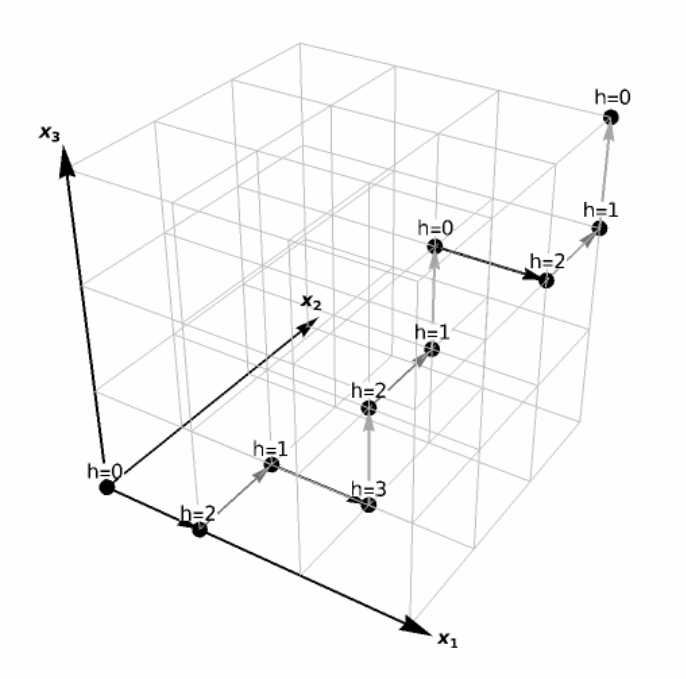}
        \caption{A 3-dimensional Balanced ballot path from $(0,0,0)$ to $(3,3,3)$ with the heights of each point along the path indicated. We use the formula \(h_3(x) = x_1 - x_2 + x_1 - x_3\) to calculate the heights.} 
        \label{fig:3ddyck}
    \end{figure}
\end{example}

\begin{definition}
    \label{def:bounded-and-height}
    For integers $k \geq 2$ and $s \geq 0$, we define the \textbf{$n$-th $k$-dimensional $s$-bounded Catalan number}, denoted by $C_{k,s,n}$, as the number of ballot paths $P$ of $kn$ steps starting at the origin $(0,0,\ldots, 0)$ and ending at $(n,n,\ldots, n)$, satisfying the following condition: for any $x = (x_1, \dots, x_k)$ in $P$, the height function $h_k(x)$ as in Definition \ref{defn:k-dims-boundCat} is less than or equal to $s$.
\end{definition}

A visualization of Balanced ballot paths for the bounded Catalan number is as follows. These are ballot paths from $\vec{0}$ to $(n,\ldots ,n)$ such that each node is between the hyperplanes $$ x_1 = \frac{x_2 + \cdots + x_k}{k-1} \mbox{ and } x_1 = \frac{x_2 + \cdots + x_k + s}{k-1}. $$ 
We consider the weight function only on the positive contribution to the height function, which means that we focus only on the change of the $x_1$-coordinate. 
\begin{definition}
\label{defn:weightedkdim}
    Given a sequence of integers $\vec{b} =(b_0,b_1,b_2,\ldots)$ and a $k$-dimensional Balanced ballot path $P_k$ of $kn$ steps, the \textbf{weight of path $P_k$ with respects to $\vec{b}$}, denoted by $wt_{\vec{b}}(P_k)$, is the product $b_{h_1} b_{h_2} \cdots b_{h_n}$, where $h_i$ is the height $($as in Definition \ref{def:bounded-and-height}$)$ of the starting point of the $i$-th up-step of $P_k$. The corresponding \textbf{$n$-th $k$-dimensional weighted Catalan number} is  $$C_{k,n}^{\vec{b}} = \sum_P wt_{\vec{b}}(P_k),$$
    where the sum is over all $k$-dimensional ballot Paths $P_k$ of $kn$ steps.

\end{definition}

\begin{definition}
\label{defn:k-dims-boundCat}
    Let $k$ be an integer at least $2$ and $s$ be a nonnegative integer. For fixed sequence of integers $\vec{b} =(b_0,b_1,\ldots)$, the \textbf{$k$-dimensional $s$-bounded weighted Catalan numbers} $C^{\vec{b}}_{k,s,n}$ are defined analogously.
\end{definition}

\begin{remark}
    The unweighted version $C_{k,s,n}$ from Definition \ref{def:bounded-and-height} corresponds to $C^{\vec{b}}_{k,s,n}$ for the weight vector $\vec{b} = (1,1,\ldots,1,0,0,\ldots)$, where the first $s+k-1$ entries are equal to $1$ and the rest are zero. 
\end{remark}

\subsection{Prior and Basic Results on Weighted 2-Catalan Numbers}
To provide a foundation for examining periodicity, we first discuss basic results on the weighted (2-dimensional) Catalan numbers, which inspire our approach to studying multidimensional weighted Catalan numbers.

Analogously to An \cite{ansthesis} and Shader \cite{shader2013weighted}, we derive a tridiagonal matrix-based recurrence for the $2$-dimensional weighted Catalan numbers. For the next preliminary result we denote by $C^{\vec{b}}_{n,i}$ the number of weighted Dyck paths from $(0,i)$ to $(2n,0)$. $($In particular, $C^{\vec{b}}_{n,0} = C^{\vec{b}}_{n})$

\begin{lemma} The $2$-dimensional weighted Catalan numbers satisfy the following recurrence:
\label{lemma:transmatrixcat}
 \[\begin{bmatrix}
C^{\vec{b}}_{n,0} \\
C^{\vec{b}}_{n,2} \\
C^{\vec{b}}_{n,4} \\
\vdots
\end{bmatrix} = 
\begin{bmatrix}
b_0 & b_0 b_1 & 0 & \dots  &\dots &\dots \\
1 & b_1 + b_2 & b_2 b_3 & 0 & \dots &\dots \\
0 & 1 & b_3 + b_4 & b_4 b_5 & 0 & \dots \\
\vdots & \vdots & \vdots & \vdots & \ddots &\dots
\end{bmatrix}
\begin{bmatrix}
C^{\vec{b}}_{n-1,0} \\
C^{\vec{b}}_{n-1,2} \\
C^{\vec{b}}_{n-1,4} \\
\vdots
\end{bmatrix}.
\] 
\end{lemma}
\begin{proof}
    We have $C^{\vec{b}}_{n,0} = b_0C^{\vec{b}}_{n-1,0}+b_0b_1C^{\vec{b}}_{n-1,2}$ and more generally $$C^{\vec{b}}_{n,2i} = C^{\vec{b}}_{n-1,2i-2} + (b_{2i}+b_{2i-1})C^{\vec{b}}_{n-1,2i} + b_{2i}b_{2i+1}C^{\vec{b}}_{n,2i+2},$$ due to the possible two steps we can take from the previous states and their corresponding weights. 
\end{proof}
\begin{remark}
\label{remark:firstentry}
    Denote the transition matrix as $A$ and note that $(C_{0,0}, C_{0,2}, C_{0,4}, \ldots) = (1, 0, 0, \ldots)$. We obtain $(C_{n,0}, C_{n,2}, C_{n,4}, \ldots) = A^n\cdot(1,0,0,\ldots) = ([A^n]_{1,1}, [A^n]_{2,1}, \dots)$. In particular, we hope to be able to efficiently compute the first entry of $A^n$ when needed. 
\end{remark}

Using this argument, we can rederive the following result from Gao and Gu \cite{gaogu}:
\begin{theorem}
    \label{theorem:general-periodicity}
    For any positive integer $m$, the sequence $C^{\vec{b}}_{1}, C^{\vec{b}}_2, \dots$ is periodic modulo $m$ if $m$ divides $b_0b_1\ldots b_k$ for some non-negative integer $k$. 
\end{theorem}

\begin{proof}
Recall that $C^{\vec{b}}_{n,i}$ is the number of weighted Dyck paths from $(0,i)$ to $(2n,0)$. Observe that the weight of each Dyck path for $i > \left\lfloor \frac{k}{2} \right\rfloor$ is divisible by $b_0b_1\ldots b_k$. Together with Lemma \ref{lemma:transmatrixcat}, this implies that the transition matrix modulo $m$ is of finite size $(\ell + 1) \times (\ell + 1)$, which depends on the parity of $k$. Specifically, $\ell = \lfloor \frac{k}{2} \rfloor.$ If $k$ is even, then the last element of the matrix is $a_{2\ell} = b_{k-1}$, and if $k$ is odd, then it is ${a_{2\ell}} = b_{k-2}+b_{k-1}$. 

\[\begin{bmatrix}
C^{\vec{b}}_{n,0} \\
C^{\vec{b}}_{n,2} \\
C^{\vec{b}}_{n,4} \\
\vdots \\
C^{\vec{b}}_{n,2 \ell}
\end{bmatrix} = 
\begin{bmatrix}
b_0 & b_0 b_1 & 0 & \dots  &\dots &\dots \\
1 & b_1 + b_2 & b_2 b_3 & 0 & \dots &\dots \\
0 & 1 & b_3 + b_4 & b_4 b_5 & 0 & \dots \\
\vdots & \vdots & \vdots & \vdots & \ddots &\dots \\
0 & 0 & \dots & 0 & 1 & a_{2 \ell}
\end{bmatrix}
\begin{bmatrix}
C^{\vec{b}}_{n-1,0} \\
C^{\vec{b}}_{n-1,2} \\
C^{\vec{b}}_{n-1,4} \\
\vdots \\
C^{\vec{b}}_{n-1,2 \ell}
\end{bmatrix}.
\]

There exist positive integers $s$ and $t$ such that $(C^{\vec{b}}_{s,0}, \ldots, C^{\vec{b}}_{s,2\ell}) \equiv  (C^{\vec{b}}_{s+t,0}, \ldots, C^{\vec{b}}_{s+t,2\ell})$, because by the Pigeonhole principle, there are at most $m^{\ell+1}$ possible combinations of $(C^{\vec{b}}_{n,0}, \ldots, C^{\vec{b}}_{n,2\ell}) \pmod m$. The matrix is of finite size, and thus the sequence will eventually be periodic, with $C^{\vec{b}}_{n+jt,0} \equiv C^{\vec{b}}_{n,0} \pmod m$ for any positive integer $j$. \end{proof}

\section{Periodicity of multidimensional Catalan numbers}\label{sec:periodicity}
Here we derive general results on the periodicity of the multidimensional weighted Catalan numbers from Definition \ref{defn:weightedkdim}, starting with the bounded ones.
\begin{proposition}
    For $k \geq 2$, every nonzero-length $k$-dimensional Balanced ballot path must reach height $k-1$.
\end{proposition}

\begin{proposition}
    For $k \geq 2$, the $(k-1)$-bounded, $k$-dimensional Catalan number is always $1$.
\end{proposition}\begin{proof}
    For every $n \geq 1$, there is one and only one ballot path from $\vec{0}$ to $(n, n, n, \dots, n)$ that does not exceed height $k-1$: it is the path described by the sequence of steps $\vec{e}_1, \vec{e}_2, \dots, \vec{e}_k$ of length $k$ repeated $n$ times.
\end{proof}
\begin{theorem}
\label{thm:periodboundedkdim}
    The sequence of the $k$-dimensional $s$-bounded and weighted Catalan numbers is periodic modulo any positive integer $m$.
\end{theorem}

\begin{proof}
    We proceed as in Theorem \ref{theorem:general-periodicity}. The weight vector is $\vec{b} = (b_0,b_1,\ldots,b_s,0,\ldots)$. Because of the height restriction, we have finitely many states $(A_n, A'_n,\ldots A^{(\ell)}_n)$. Then the transition matrix for $C^{k,\vec{b}}_n$ is of finite size $\ell \times \ell$. There are at most $m^{\ell+1}$ possible combinations of $(A_n, A'_n,\ldots A_n^{(\ell)}) \pmod m$, hence by the Pigeonhole principle, there exist positive integers $s$ and $t$ such that $(A_s,A'_s,\ldots A_s^{(\ell)}) \equiv (A_{s+t},A'_{s+t},\ldots A^{(\ell)}_{s+t}) \pmod m$. Thus, $A_{s+pt} \equiv A_{s} \pmod m$ for any positive integer $p$ and the sequence is eventually periodic.
\end{proof}
\begin{corollary}
    For any fixed positive integers $k \geq 2$ and $h$, the sequence $D^k_{n,h}$, denoting the $k$-dimensional Balanced ballot paths of height $h$, is periodic modulo 
    any positive integer $m$.
\end{corollary}
\begin{proof}
    In the next Section, we show that $D_{k,h,n} = C_{k,h,n} - C_{k,h-1,n}$. From Theorem \ref{thm:periodboundedkdim} it follows that both $C_{k,h,n}$ and $C_{k,h-1,n}$ are periodic modulo $m$. It remains to note that if two sequences are periodic modulo $m$, then their difference is also periodic and additionally $p_m(D_{k,h,n}) = \operatorname{lcm}(p_m(C_{k,h,n}), p_m(C_{k,h-1,n}))$, where $p_m$ denotes period modulo $m$.
\end{proof}

We obtain an analogous result to Theorem \ref{theorem:general-periodicity} on the periodicity of the $k$-dimensional weighted Catalan numbers.
\begin{theorem}
\label{theorem:period-weighted-kdim}
    For any positive integer $m$ and a weight vector $\vec{b}$, the sequence $C^{\vec{b}}_{k,1}, C^{\vec{b}}_{k, 2}, \dots$ is eventually periodic modulo $m$ if there exists a positive integer $s$, such that each of the weights $b_s,b_{s+1},\ldots b_{s+k-2}$ is divisible by $m$.
\end{theorem}
\begin{proof}
    The weight of the path changes only at each up-step (see Definition \ref{defn:weightedkdim}). We observe what happens after one up-step. For the path to reach a height greater than $s+k-2$, all steps should start at a point with height between $h_{s}, h_{s+1},\ldots h_{s+k-2}$, as an up-step changes the height by $k-1$. All weights at these heights are divisible by $m$, and thus for each $k$-dimensional Balanced ballot path $P_k$ with $h_k(x)>s+k-2$, the weight $wt_{\vec{b}}(P_k) \equiv 0 \pmod m$. Then it is enough to consider only the paths with $h_k(x) \leq s+k-2$, i.e., the $(s+k-2)$-bounded ballot paths. By Proposition \ref{thm:periodboundedkdim}, their corresponding Catalan numbers $C^{k,s,\vec{b}}_n$ are periodic modulo $m$.
\end{proof}

Similar statements can be proven when $m$ divides the product of several weights. However, the greater the number of weights, the more of their permutations are divisible by $m$ we obtain as a requirement. Here are more specific conditions for the scenario in three dimensions.

\begin{theorem}
 For any positive integer $m$ and sequence of integers $\vec{b}$, the sequence $C^{\vec{b}}_{k,1}, C^{\vec{b}}_{k,2}, \dots$ is eventually periodic modulo $m$, if there exists a positive integer $s$ such that $m\mid b_{s-j}b_{s+k-j'}$ for all $j \in \{0,1,2 \dots, k-1\}, j' \in \{j,j+1, \dots k-1\}$.
\end{theorem}
\begin{proof}
  We contend that the last two steps in the $x_1$ direction before the path is above height $s+k$ are always of the following form: a step in the $x_1$ direction from height $s-j$ for some $j \in \{0, 1, 2, ..., k-1\}$ to height $s-j+k$ and then a step in the $x_1$ direction from height $\ell$ for some $\ell \in \{s+1, s+2, \dots, s+k-j'\}$ to $\ell+k$.

  Therefore we find that any summand in $C^{\vec{b}}_{k, n}\pmod{m} = \sum_{P}wt_{\vec{b}}(P)$ from any path that exceed height $s+k$ is $0 \pmod{m}$. Therefore we find that $C^{\vec{b}}_{k, n}\pmod{m} = C^{\vec{b}'}_{k, n}\pmod{m}$ where $b'_i = b_i$ for $i \in \{0,1, \dots, s\}$ and $b'_i = 0$ everywhere else. We know from Theorem~\ref{thm:periodboundedkdim} that $C^{\vec{b}'}_{k, n}\pmod{m}$ is eventually periodic. This completes the proof.
\end{proof}

\section{Examples of Weighted $k$-dimensional Catalan Numbers}
\subsection{Recursive formula for certain higher-dimensional weighted and bounded Catalan numbers}
\label{sec:multidimensional}

Here we obtain formulas for specific sequences of the $k$-dimensional $s$-bounded and weighted numbers $C^{\vec{b}}_{k,s,n}$. We begin with a general $k$, but later focus mostly on $k=3$. 
\begin{theorem}
\label{theorem:kboundedkdimensional}
    The $k$-dimensional $k$-bounded and weighted Catalan numbers satisfy the recurrence
    $$C^{\vec{b}}_{k,k,n}= (b_0 + (k-1)b_1)C^{\vec{b}}_{k,k,n-1}.$$ 
\end{theorem}

\begin{proof}
    For clarity, denote $A_n = C^{\vec{b}}_{k,k,n}$ and let $B_{n-1}$ be the number of paths from $(2,1,1,\ldots,1,1,0)$ to $(n,n,\ldots,n)$, such that for each node $(v_1,\ldots,v_k)$ we have $v_1 \geq v_2 \geq \cdots \geq v_k$ and $h(x) \leq k$.

    By the definition of a Balanced ballot path, there is only one sequence of $k$ steps from $(a,\ldots, a)$ to $(a+1, \ldots, a+1)$ with a weight contribution of $b_0$. Due to height restrictions, there is only one sequence with $k$ steps from $(a,\ldots, a)$ to $(a+2, a+1, \ldots, a+1,a)$ with a weight contribution of $b_0b_1$. There are $k-1$ ways to go from $(a,a-1,\ldots,a-1,a-2)$ to $(a+1,a,\ldots,a,a-1)$ each with contribution of $b_1$, because the $\vec{e}_1$ step should be the last one and there are $k-1$ possibilities for when the $\vec{e}_k$ step will occur. Similarly, there are $k-1$ ways to go from $(a,a-1,\ldots,a-1,a-2)$ to $(a,\ldots, a)$ with weight contribution of $1$. Using these relations, we obtain the recurrence:
    
    
     \[ \begin{bmatrix}
A_{n}\\
B_{n}
\end{bmatrix} = \begin{bmatrix}
b_0 & b_0b_1 \\
k-1 & (k-1)b_1 \end{bmatrix}
\begin{bmatrix} A_{n-1} \\ B_{n-1} \end{bmatrix}. \]

From $A_n = b_0A_{n-1} + b_0b_1B_{n-1}$ it follows $\displaystyle B_{n-1} = \frac{A_n -b_0A_{n-1}}{b_0b_1}$. Substituting into the second row we get $\displaystyle B_n = (k-1)A_{n-1} + (k-1)\frac{A_n -b_0A_{n-1}}{b_0} = \frac{k-1}{b_0}A_n$. Hence 
\vspace{-0.5 em} $$A_n = (b_0 + (k-1)b_1)A_{n-1}. \qedhere$$ 
\end{proof}
\vspace{-0.5 em}
From the recurrence relation for $A_n$ and weights $b_0=b_1=1$ it directly follows that:
\begin{corollary*}
    The $k$-dimensional $k$-bounded Catalan numbers satisfy $C_{k,k,n} = k^{n-1}$.
\end{corollary*}

For the next results, given a value of $s$, we denote by $A_n$ the number of $3$-dimensional bounded Balanced ballot paths from $(0,0,0)$ to $(n,n,n)$, by $B_{n-1}$ the number of paths from $(2,1,0)$ to $(n,n,n)$, by $C_{n-2}$ the number of paths from $(3,3,0)$ to $(n,n,n)$, by $D_{n-1}$ the number of paths from $(3,0,0)$ to $(n,n,n)$, by $E_{n-2}$ the number of paths from $(4,2,0)$ to $(n,n,n)$, and by $F_{n-3}$ the number of paths from $(5,4,0)$ to $(n,n,n)$. Each of them, satisfies that for each node $x = (x_1,x_2,x_3)$ we have $x_1 \geq x_2 \geq x_3$ and $h(x) \leq s$. Because we start from the origin, $(A_0,B_0,C_0, \ldots) = (1,0,0, \ldots)$. Note that, by definition, $A_n = C^{\vec{b}}_{3,s,n}.$

Because of the height restriction, the weight vector for the  $3$-dimensional $4$-bounded Catalan numbers is $\vec{b}=(b_0,b_1,b_2,0,\ldots)$.


\begin{proposition}
\label{proposition:4bound3d}
    The $3$-dimensional $4$-bounded and weighted Catalan numbers satisfy the recurrence
     \[\begin{bmatrix}
A_n \\
B_n \\
C_n
\end{bmatrix} =
\begin{bmatrix}
b_0 & b_0 b_1 + b_0 b_2 & 0 \\
2 & 2 b_1 + 2 b_2 & b_2 \\
1 & b_1 + b_2 & b_2
\end{bmatrix}
\begin{bmatrix}
A_{n-1} \\
B_{n-1} \\
C_{n-1}
\end{bmatrix}.
\]
\end{proposition}

\begin{proof}
    As in Theorem \ref{theorem:kboundedkdimensional}, we observe the possibilities after every $3$ steps.
     By the definition of a Balanced ballot path, there is one way to go from a point of form $(a, a, a)$ to $(a+1,a+1,a+1)$, namely $\vec{e_1},\vec{e_2},\vec{e_3}$ with a weight contribution of $b_0$. Due to the height restriction there are two sequences of steps that one can take from $(a, a, a)$ to $(a+2,a+1, a)$, namely $\vec{e_1},\vec{e_2},\vec{e_1}$ and $\vec{e_1}, \vec{e_1},\vec{e_2}$ with weight contributions of $b_0b_1$ and $b_0b_2$. Likewise, the weight contributions for the steps starting from $(a,a-1,a-2)$ to $(a+1,a,a-1)$ are $2b_1 + 2b_2$, to $(a+1,a+1,a-2)$ is $b_2$, and to $(a,a,a)$ is $1.$ Finally, the weight contribution for the way from $(a,a,a-3)$ to $(a,a,a)$ is $1$, to $(a+1,a,a-1)$ is $b_2 + b_1$, and to $(a+1,a+1,a-2)$ is $b_2$. The possibilities are displayed in Figure \ref{fig:states3ddyck}. 
    \qedhere
    \begin{figure}[h]
    \centering
    \includegraphics[width=0.85\linewidth]{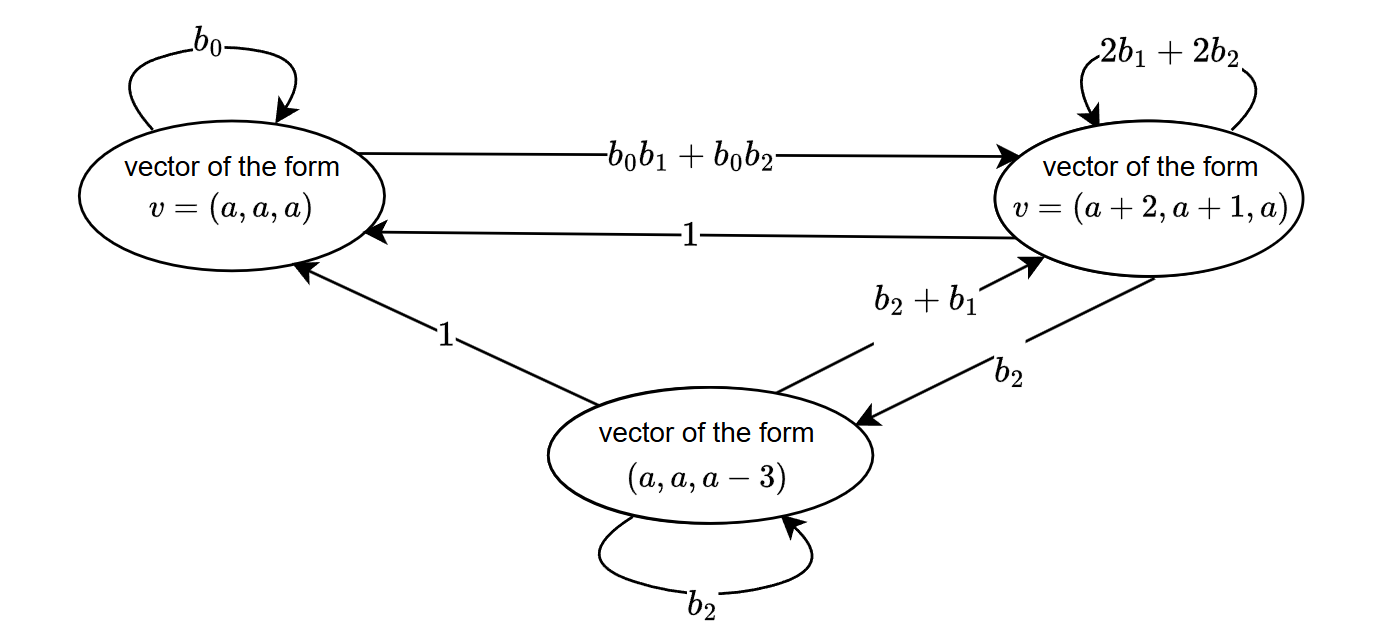}
    \caption{The possible states for $C^{\vec{b},3}_n$} with $\vec{b}=(b_0,b_1,b_2, 0, 0 ,\ldots)$
    \label{fig:states3ddyck}
\end{figure}
    \end{proof}

In the unweighted case, $\vec{b} = (1,1,1,0,\ldots)$ computations from the matrix give $A_n = 6A_{n-1} - 3A_{n-2}$. The sequence with this recurrence is $A158869$ in OEIS, \cite{oeis} and also counts the number of ways of filling a $2 \times 3 \times 2n$ parallelepiped with $1 \times 2 \times 2$ bricks.
\begin{corollary}
    The $3$-dimensional $4$-bounded unweighted Catalan numbers satisfy the recurrence $$C_{3,4,n} = 6C_{3,4,n-1} - 3C_{3,4,n-2}$$
\end{corollary}

\noindent Similarly, for $s=5$:

\begin{proposition}
    The $3$-dimensional $5$-bounded Catalan numbers $C^{\vec{b}}_{3,5,n}$ satisfy the recurrence 
     \[\begin{bmatrix}
A_n \\
B_n \\
C_n
\end{bmatrix} =
\begin{bmatrix}
b_0 & b_0b_2+b_0b_1 & 0 \\
2 & 2(b_1+b_2+b_3) & b_3+b_2 \\
1 & b_3+b_2+b_1 & 2(b_3+b_2+b_1)
\end{bmatrix}
     \begin{bmatrix}
A_{n-1} \\
B_{n-1} \\
C_{n-1}
\end{bmatrix}.
\]

\end{proposition}
\begin{proof}
  As in Proposition \ref{proposition:4bound3d} we have three states. Due to the height restriction there is one way to go from $(a,a,a)$ to $(a-1,a-1,a-1)$, two ways from $(a,a,a)$ to $(a+2,a+1,a)$, six from $(a,a-1,a-2)$ to $(a+1,a,a-1)$, two to $(a,a,a)$ and two to $(a+1,a+1,a-2)$. Finally, the number of ways to go from $(a,a,a-3)$ to $(a+1,a,a-1)$ are $3$, to $(a+1,a+1,a-2)$ are $3$ and to $(a,a,a)$ is one. Considering the weights at each up-step we obtain the recurrence.
\end{proof}

For $s=6$ we have many more states and we give a derivation only in the unweighted case.

\begin{proposition}
    
    The $3$-dimensional $6$-bounded Catalan numbers $C_{3,6,n}$ satisfy the following recurrence:
    \[\begin{bmatrix}
A_n \\
B_n \\
C_n \\
D_n \\
E_n \\ 
F_n 
\end{bmatrix} =
\begin{bmatrix}
 1 & 2 & 0 & 1 & 0 & 0\\
2 & 6 & 2 & 1 & 1 & 0 \\
1 & 3 & 3 & 0 & 2 & 2 \\
0 & 2 & 1 & 3 & 3 & 0 \\
0 & 3 & 3 & 0 & 2 & 0 \\
0 & 1 & 3 & 0 & 1 & 2
\end{bmatrix}
\begin{bmatrix}
A_{n-1} \\
B_{n-1} \\
C_{n-1} \\
D_{n-1} \\
E_{n-1} \\ 
F_{n-1} 
\end{bmatrix}.
\]
\end{proposition}

\begin{proof}
The result is obtained in the same manner as Proposition \ref{proposition:4bound3d}, but with $6$ states not $3.$
\end{proof}


\subsection{Multidimensional Balanced-Ballot-Path-Height triangles}\label{sec:heighttriangles}

For the next results, we used a Python program to determine each $C_{k,s,n}$. 

Denote by $D_{k,s,n}$ the number of $k$-dimensional balanced ballot paths of $kn$ steps such that height is exactly $s$, i.e. for at least one intermediate point $h(x)=s$, but for no points $h(x) > s$.

In the table below, the numbers in each row correspond to the number of Balanced ballot paths of $kn$ steps and height from $k-1$ to $(k-1)n$. This is similar to the $2$-dimensional Balanced-ballot-path-height triangle (sequence $A080936$ in OEIS \cite{oeis}), but for $k=3$. Recall Definition \ref{def:bounded-and-height} for the $k$-dimensional $s$-bounded Catalan numbers $C_{k,s,n}$. It is evident that $D_{k,s,n} = C_{k,s,n} - C_{k,s-1,n}$.
\vspace{-0.8 em}
\begin{table}[h]
    \centering
  \[
\begin{array}{c|ccccccccccc} n\backslash h & 2 & 3 & 4 & 5 & 6 & 7 & 8 & 9 & 10 & 11 & 12\\
\hline
1 & 1 \\
2 & 1 & 2 & 2 \\
3 & 1 & 8 & 18 & 10 & 5 \\
4 & 1 & 26 & 120 & 142 & 117 & 42 & 14 \\
5 & 1 & 80 & 720 & 1481  & 1789 & 1130 & 596 & 168 & 42 \\
6 & 1 & 242 & 4122 & 13680 & 23205 & 20940 & 14 817 & 6936 & 2781 & 660 & 132
\end{array}
\]
    \caption{The $3$-dimensional Ballanced-Ballot-Path-Height Triangle}
    \label{table:3dtriangle}
\end{table}

Table \ref{table:3dtriangle} gives the first $36$ numbers of the triangular array sequence $D_{3,s,n}$. Note that the sum of the numbers of every $n$-th row is equal to the $n$-th $3$-dimensional Catalan number (sequence $A005789$ in the OEIS \cite{oeis}). Moreover, for any $n$ we have $D_{3,2n,n} = C_n$, where $C_n$ is the classical $n$-th Catalan number. Indeed, the first $n$ steps should be in the direction of $\vec{e_1}$, to reach height $2n$, and the number of the remaining steps, i.e., ballot paths of length $2n$ formed by $n$ steps of $\vec{e_2}$ and $n$ steps of $\vec{e_3}$, is the $n$-th $2$-dimensional Catalan number. 

\begin{table}[h]
    \centering
 \[
\begin{array}{c|ccccccccccc} n \backslash h & 3 & 4 & 5 & 6 & 7 & 8 & 9 & 10 & 11 & 12 \\
\hline
1 & 1  \\
2 & 1 & 3 & 5 & 5 \\
3 & 1 & 15 & 68 & 147 & 105 & 84 & 42 \\
4 & 1 & 63 & 722 & 3098 & 4720 & 5940 & 5112 & 2520 & 1386 & 462\\
\end{array}
\] 

    \caption{The $4$-dimensional Balanced-Ballot-Path-Height triangle}
    \label{table:4dtriangle}
\end{table}

 Table \ref{table:4dtriangle} gives the first $22$ elements of the sequence $D_{4,s,n}$. The sum of the numbers of every $n$-th row is equal to the $n$-th $4$-dimensional Catalan number (sequence $A005790$ in the OEIS, \cite{oeis}). Moreover, for all $n$ we have $D_{4,3n,n} = C_{3,n}$, again by arguments, analogous to the $3$-dimensional case. This suggests that generalizations are possible.
\smallskip

For each row of the $k$-dimensional Balanced-ballot-path-height triangle we have: 
$$\sum_{s=k-1}^{n(k-1)} D_{k,s,n} = C_{k,n} \ \ \text{and} \ \ D_{k,s,n} = C_{k,s,n} - C_{k,s-1,n}.$$
\begin{proposition}
\label{proposition:dyckheight-catalan}
For a $k$-dimensional Balanced-ballot-path-height triangle we have
\[D_{k,(k-1)n,n} = C_{k-1,n}.\]
\end{proposition}
\begin{proof}
   The sequence $D_{k,(k-1)n,n}$ counts the $k$-dimensional Balanced ballot paths of length $kn$, with height $(k-1)n$. The only way to reach this height is when the first $(k-1)$ steps are all up-steps of $\vec{e_1}$ -- otherwise, if there were $t$ steps of $\vec{e_i}$ in between them, then for the height function we obtain $h(x) = (k-1) - t < k-1$, contradiction. Therefore, the first $k-1$ steps are all $\vec{e_1}$ and the number of $k$-dimensional ballot paths starting from $(k-1,0,0,\ldots,0)$ and ending at $(k-1,k-1,\ldots,k-1)$ equals the number of $(k-1)$-dimensional ballot paths from $(0,\ldots,0)$ to $(k-1,\ldots, k-1)$ which is $C_{k-1,n}$.
\end{proof}

\begin{proposition}
    For a $k$-dimensional Balanced-ballot-path-height triangle, we have
    \[D_{k,(k-1)n-1,n} = (n-1)D_{k,(k-1)n,n}.\]
\end{proposition}
\begin{proof}
    A $k$-dimensional Balanced ballot path has height $(k-1)n-1$ only if the first $(k-1)n+1$ elements consist of $(k-1)n$ steps in the $\vec{e}_1$ direction and one step in the direction of $\vec{e}_2$. The number of ways for this is $(n-1)$. Let $S_1$ be the collection of pairs whose first entry consists of a $k$-dimensional ballot Path of $kn$ steps and maximum height $(k-1)n$, and whose second entry consists of an integer in $\{2,3,\ldots,n\}$. Let $S_2$ be the collection of $k$-dimensional Balanced ballot Paths of $kn$ steps and height $(k-1)n-1$. It suffices to show $|S_1| = |S_2|$. We illustrate a bijection between $S_1$ and $S_2$ -- given a pair $(P, i)$, we construct a ballot Path $P'$ as follows. The $i$-th step of $P'$ is in the direction $\vec{e}_2$, while all steps from the first to the $(n+1)$-st one, except for the $i$-th one, are in the direction $\vec{e}_1$, and the directions of all of the other steps in $P'$ are the same as those of $P$.
\end{proof}
\subsection{A multidimensional generalization of the Narayana triangle}\label{sec:Narayana}
Here we consider another statistic of the $k$-dimensional ballot paths, namely the number of peaks. A peak is a node, to which we have arrived with an up-step and left from with a down-step. For example, the path on Figure \ref{fig:dyckpath} has three peaks. In the $2$-dimensional case, the sequence representing the number of paths of length $n$ with a fixed number of peaks is called the Narayana numbers ($A001263$ in OEIS \cite{oeis}). There is a higher-dimensional analog \cite{multidimnarayana}, where a peak is a node, to which we have arrived with an $\vec{e_i}$  step and left from with an $\vec{e_j}$ step for some $i<j$; by our definition of peak, we provide an alternative.

In the context of $k$-dimensional Balanced ballot paths, increases in the first coordinate \( x_1 \) represent a positive change in height at points. Therefore, we consider it the primary coordinate. This is why, for any $k \geq 2$ and $k$-dimensional ballot path $P$, we consider $\vec{e}_1$ to be an up-step.

\begin{definition}
Denote by $N_{k,m,n}$ the number of $k$-dimensional ballot Paths with $p$ peaks, where a peak is a node, to which we have arrived with an $\vec{e_1}$ step and left with an $\vec{e_j}$ step for some $j > 1$.

\end{definition}
We wrote Python code to generate the first few elements of $N_{3,p,n}$ and $N_{4,p,n}$. Tables \ref{table:3dnarayana} and \ref{table:4dnarayana} show respectively the $3$-dimensional and $4$-dimensional Narayana triangles for our height $h$. The $3$-dimensional one has a corresponding sequence $A338403$ in OEIS \cite{oeis}, with another combinatorial interpretation -- counting the number $(n,k)$-\textit{Duck words} \cite{duckwords}. On the other hand, the $4$-dimensional one does not seem to be currently available at OEIS. Note that for each row of the $k$-dimensional Narayana triangle we have:
$$\sum_{h=k-1}^{n(k-1)} N_{k,p,n} = C_{k,n} =\sum_{h=k-1}^{n(k-1)} D_{k,s,n}$$
Another property of the $k$-dimensional Narayana triangle is that the numbers in the first column form the sequence of $(k-1)$-dimensional Catalan numbers and hence also connects with the Balanced-ballot-path-height triangle. Together with Proposition \ref{proposition:dyckheight-catalan}, this yields a relation between the three types of sequences. 
\begin{proposition}
\label{proposition:Narayana-Dyck-Cat}
    For a $k$-dimensional Narayana triangle we have: $N_{k,1,n} = C_{k-1,n} = D_{k,(k-1)n,n}.$
\end{proposition}
\begin{proof}
    In order to have only one peak, there should be only one pair of steps $\vec{e_1},\vec{e_i}$. The only way for the condition to be satisfied is if the first $n$ steps are up-steps of $\vec{e_1}$. These paths are the same as in Proposition \ref{proposition:dyckheight-catalan}. The number of $k$-dimensional ballot Paths starting from $(k-1,0,0,\ldots,0)$ and ending at $(k-1,k-1,\ldots,k-1)$ equals the number of $(k-1)$-dimensional ballot paths from $(0,\ldots,0)$ to $(k-1,\ldots, k-1)$ which is $C_{k-1,n}$.
\end{proof}
\begin{table}

\[
\begin{array}{c|ccccccc} 
n \backslash p & 1 & 2 & 3 & 4 & 5 & 6\\
\hline
1 & 1 &  &  &  \\
2 & 2 & 3 &  &  \\
3 & 5 & 23 & 14 &  \\
4 & 14 & 131 & 233 & 84 \\
5 & 42 & 664 & 2339 & 2367 & 594 \\
6 & 132 & 3166 & 18520 & 36265 & 24714 & 4719 \\\end{array}
\]
\caption{Table of values of $N_{3, p, n}$, the $3$-dimensional Narayana triangle.}\label{table:3dnarayana}
\end{table}

\vspace{-1.0 em}
\vspace{-0.8 em}

\begin{table}
\[
\begin{array}{c|ccccccc} 
n \backslash p & 1 & 2 & 3 & 4 & 5 & 6\\
\hline
1 & 1 &  &  &  \\
2 & 5 & 9 &  &  \\
3 & 42 & 236 & 184 &  \\
4 & 462 & 5354 & 12268 & 5940 \\
5 & 6006 & 118914 & 543119 & 737129 & 257636 \\
6 & 87516 & 2653224 & 20245479 & 53243052 & 50245691 & 13754842 \\
\end{array}
\]
\caption{Table of values of $N_{4, p, n}$, the $4$-dimensional Narayana triangle.}
\label{table:4dnarayana}
\end{table}

\section{Acknowledgments}
This project was started during the Research Science Institute (RSI) 2025, hosted at the Massachusetts Institute of Technology (MIT). We greatly appreciate Dr. Tanya Khovanova and Professor Alexander Postnikov for insightful discussions. We also thank supervisors Prof. Roman Bezrukavnikov and Dr. Jonathan Bloom for their general advice and recommendations. We also acknowledge AnaMaria Perez, Dr. Jenny Sendova, Miroslav Marinov, Prof. Stanislav Harizanov, Nick Arosemena, Mircea Dan Hernest, and Austin Luo for their feedback. We thank the Center for Excellence in Education and MIT for allowing us to work on this project during the Research Science Institute $2025$. The first author is employed by the MIT Department of Mathematics. The second author is supported by the St. Cyril and St. Methodius International Foundation, the EVRIKA Foundation and the High School Student Institute of Mathematics and Informatics in Bulgaria.

Figures~\ref{fig:dyckpath}-~\ref{fig:3ddyck} were generated using TikZ. Figure~\ref{fig:states3ddyck} was generated using \href{draw.io}{draw.io}.


\bibliographystyle{plain}

\end{document}